\documentclass[12pt]{article}
\usepackage{amssymb,amsthm,latexsym}
\usepackage{amsmath}

\newtheorem{thm}{Theorem}[section]
\newtheorem{cor}{Corollary}
\newtheorem{lem}[thm]{Lemma}

\theoremstyle{definition}

\newcommand{\allone}{\mathbf{1}}
\newcommand{\ZZ}{\mathbf{Z}}
\newcommand{\ccode}{\mathcal{C}}
\DeclareMathOperator{\ch}{char}

\DeclareMathOperator{\kernel}{Ker}
\DeclareMathOperator{\diag}{diag}
\DeclareMathOperator{\Diag}{Diag}
\DeclareMathOperator{\Sym}{Sym}
\DeclareMathOperator{\Alt}{Alt}

\DeclareMathOperator{\rank}{rank}
\DeclareMathOperator{\wt}{wt}
\DeclareMathOperator{\Aut}{Aut}

\title{Mass Formula for Self-Orthogonal Codes over $\mathbf{Z}_{p^2}$}

\date{December 23, 2008}

\begin{document}
\maketitle

\begin{center}
Dedicated to Professor D. K. Ray-Chaudhuri on the occasion of
his 75th birthday
\end{center}

\centerline{\scshape Rowena Alma L. Betty
\footnote{On study leave from the Institute of Mathematics, 
University of the Philippines-Diliman, Quezon City 1101 Philippines}}
\medskip
{\footnotesize 
\centerline{Graduate School of Information Sciences}
\centerline{Tohoku University}
\centerline{Sendai 980--8579, Japan}
} 

\medskip

\centerline{\scshape Akihiro Munemasa}
\medskip
{\footnotesize 
\centerline{Graduate School of Information Sciences}
\centerline{Tohoku University}
\centerline{Sendai 980--8579, Japan}
} 

\medskip


\begin{abstract}
In this note, we establish a mass formula for self-orthogonal
codes over $\ZZ_{p^2}$, where $p$ is a prime. As a consequence, an alternative proof of the known mass
formulas for self-dual codes over $\ZZ_{p^2}$ is obtained. We also establish a mass formula for even quaternary codes, 
which includes a mass formula for Type II quaternary codes as a special case.
\end{abstract}


\section{Introduction}\label{Sec:Intro}
In this paper, we establish a mass formula for self-orthogonal
codes over $\ZZ_{p^2}$, where $p$ is a prime. This means finding a number $M(n)$ such that
$$M(n)=\sum_{\ccode}\frac{|E|}{|\Aut\ccode|}$$ where $\ccode$ runs through the equivalence classes of self-orthogonal codes of length $n$ over $\ZZ_{p^2}$, $E$ is the full group
of all transformations that we allow in defining equivalence
for code $\ccode$, and $\Aut\ccode$ is the automorphism group of $\ccode$.
Clearly, the mass formula gives the total number of distinct self-orthogonal codes.
Mass formulas for quaternary self-dual and Type II codes were
given in \cite{Gaborit}, while for odd primes $p$, 
a mass formula for self-dual codes over $\ZZ_{p^2}$ was 
given in \cite{Z9,Zp2}. Mass formulas for self-dual codes 
over a variety of rings, taken from different references, 
can also be found in \cite{Rains-Sloane}.   

Unlike mass formulas for codes over finite fields, our mass
formula for self-orthogonal codes over $\ZZ_{p^2}$ is given
as the sum of mass formulas over some finer classes of codes.
This is because codes over $\ZZ_{p^2}$ have an invariant called
a type, denoted $\{k_1, k_2\}$, in place of the dimension
for codes over finite fields. The type of a code over $\ZZ_{p^2}$
is determined by the dimensions of its residue and torsion. We
shall determine the number of self-orthogonal codes over
$\ZZ_{p^2}$ with given residue and torsion. This number is shown
to be a power of the prime $p$, as the set of such codes has a
structure of an affine subspace of the space of matrices over
$\ZZ_p$. Moreover, this number depends only on the dimensions of
the residue and torsion, and is independent of particular 
choices of a residue and torsion. When $p=2$, there is a subclass
of the class of 
self-orthogonal codes over $\ZZ_4$ which we call the class of
even codes. When an even code over $\ZZ_4$ is self-dual, then it
is called type~II, which is commonly used terminology. In the
literature, some authors require that type~II codes over $\ZZ_4$
to contain the all-ones vector, while others do not. We shall
establish mass formulas for type~II codes for both variants.

The paper is organized as follows. Section~2 introduces necessary
terminology. In Section~3, we consider some mappings on the space
of matrices. These mappings are used in Section~4 to describe the
set of generator matrices of self-orthogonal codes of given
residue and torsion. In Section~5, we establish results analogous
to those in Section~4, for quaternary even codes. Finally in
Section~6, we give a mass formula for self-orthogonal codes of
given type, and as a corollary, we also give a mass formula
for self-dual codes. Analogous results for quaternary
even codes are also given.

\section{Preliminaries}
\label{Sec:Pre}
For a positive integer $m$, we denote by $\ZZ_m$ the ring of
integers modulo $m$.
A code $\ccode$ of length $n$
over $\ZZ_m$ is a
submodule of $\ZZ_{m}^n$. 
For a matrix $G\in M_{k\times n}(\ZZ)$, we denote by
$\ZZ_{m}^k G$ the code 
$\{aG\bmod{m} \mid a\in \ZZ^k\}$ of length $n$ over $\ZZ_m$.
A generator matrix of a code $\ccode$ of length $n$ over
$\ZZ_m$ is a matrix $G\in M_{k\times n}(\ZZ)$ such that
$\ccode=\ZZ_{m}^k G$. Usually, entries of a generator matrix
of a code over $\ZZ_m$ are taken to be in $\ZZ_m$. 
However, since we deal with codes over $\ZZ_p$ and $\ZZ_{p^2}$ at the same time, we adopt
this non-standard convention to avoid cumbersome notation.

We denote by $x\cdot y$ the standard inner product
of vectors $x,y$ in $\ZZ_{m}^{n}$, and
by $\ccode^\perp$ the dual code of a code
$\ccode$ over $\ZZ_m$ with respect to
this inner product. A code $\ccode$ is said to be
self-orthogonal (respectively self-dual) if $\ccode\subset
\ccode^\perp$ (respectively $\ccode=\ccode^\perp$) holds.
Two codes $\ccode,\ccode'$ of the same length
are said to be equivalent if there exists a $(1,-1)$
monomial matrix which maps $\ccode$ onto $\ccode'$. 

Let $p$ be a prime, and consider the exact sequence
\begin{equation*}
0\to\ZZ_{p}\stackrel{\iota}{\to}
\ZZ_{p^2}\stackrel{\pi}{\to}
\ZZ_{p}\to0,
\end{equation*}
where $\iota$ is the composition of the isomorphism
$\ZZ_p\to p\ZZ_{p^2}$ and the embedding
$p\ZZ_{p^2}\to\ZZ_{p^2}$, and $\pi$ is the
canonical homomorphism. 
For a positive integer $n$, by abuse of notation, we denote
the cartesian product of the mappings $\iota$ and $\pi$
by the same symbols:
\begin{equation*}
0\to\ZZ_{p}^n\stackrel{\iota}{\to}
\ZZ_{p^2}^n\stackrel{\pi}{\to}
\ZZ_{p}^n\to0\quad\text{(exact).}
\end{equation*}
For a code $\ccode$ over $\ZZ_{p^2}$,
$\pi(\ccode)$ is called the residue code
of $\ccode$, and $\iota^{-1}(\ccode)$ is called the
torsion code of $\ccode$. 
Since $\iota\circ\pi=p$, we have $\pi(\ccode)\subset\iota^{-1}(\ccode)$. 
Moreover, since $\mbox{Im}\;\iota = \kernel\pi$, we have  
\begin{equation}\label{cardC}
|\ccode|=|\pi(\ccode)||\iota^{-1}(\ccode)|. 
\end{equation}
Every code of length $n$ over $\ZZ_{p^2}$ is equivalent
to a code $\ccode$ with generator matrix
\[
\begin{bmatrix} 
I_{k_1}&A\\ 0&pB
\end{bmatrix}
\]
where 
$A\in M_{k_1\times (n-k_1)}(\ZZ)$, 
$B\in M_{k_2\times (n-k_1)}(\ZZ)$.
Note that
\[
\begin{bmatrix} 
I_{k_1}&A
\end{bmatrix}
\]
is a generator matrix of the residue code of $\ccode$,
and
\[
\begin{bmatrix} 
I_{k_1}&A\\ 0&B
\end{bmatrix}
\]
is a generator matrix of the torsion code of $\ccode$. We say that the code $\ccode$ has type $\{k_1, k_2\}$.

The following lemma is due to Conway and 
Sloane~\cite[p. 34]{C-S-Z4}.

\begin{lem}\label{DEres}
Let $\ccode$ be 
a self-orthogonal code of length $n$ over 
$\ZZ_4$. Then $\pi(\ccode)$ is doubly even and
$\pi(\ccode)\subset\iota^{-1}(\ccode)\subset\pi(\ccode)^\perp$
holds.
\end{lem}

We define the Euclidean weight in $\ZZ_4$ by $\wt_e (0)=0$, $\wt_e(1)=\wt
_e(3)=1$ and $\wt_e(2)=4$. The Euclidean weight of a vector $x=(x_1,\ldots ,x_n) \in \ZZ_{4}^{n}$ is defined by
$$
\wt_e(x)=\sum_{i=1}^n \wt_e(x_i).
$$
A quaternary code is said to be even if
the Euclidean weight of every codeword is divisible by $8$.
Every quaternary even code is self-orthogonal.
A quaternary self-orthogonal code with generator matrix all of
whose row vectors have Euclidean weight divisible by $8$ is even,
by \cite[Lemma 2.2]{BDHO}. If a quaternary even code contains a codeword all
of whose coordinates are $\pm1$, then the length is divisible by $8$.
A quaternary even self-dual code is also called a
quaternary Type~II code. It is known (see \cite[Lemma 2.2]{HSG})
that a quaternary Type~II code
contains a codeword all of whose coordinates are $\pm1$, but in some earlier literature (see for example \cite{Z4-B-S-B-M}), a quaternary Type II code is assumed to contain the all-ones vector $\bf{1}$.

\section{Some mappings on matrices}

Let $K$ be a field, $m$ a positive integer. We denote by 
$\Sym_{m}(K)$ the set of symmetric $m\times m$ 
matrices and $\Alt_{m}(K)$ the set of alternating 
$m\times m$ matrices over $K$. 
For a square matrix $A$ of order $n$, we denote by
$\diag(A)$ the $n$-dimensional vector composed by
the diagonal entries of $A$, and by $\Diag(A)$
the diagonal matrix whose diagonal entries are those
of $A$.
If the characteristic of $K$
is $2$, then
\[
\Alt_m(K)=\{ A\in\Sym_m(K)\mid \diag(A)=0\}.
\]

\begin{lem}\label{rankA} 
If $A\in M_{m\times n} (K)$ has rank $m$, then the 
image of the map 
$$\begin{array}{ccl}
\Psi_A  : M_{m\times n} (K) & \rightarrow & M_{m} (K) \\
N & \mapsto & AN^t + NA^t
\end{array}$$
is $\Sym_m(K)$  if $\ch K\neq 2$ and $\Alt_m(K)$ 
if $\ch K = 2$. Moreover, for $\ch K = 2$, the map 
$$\begin{array}{ccl}
\Phi_A  : M_{m\times n} (K) & \rightarrow & \Sym_m(K) \\
N & \mapsto & AN^t + NA^t + \Diag(AN^t)
\end{array}$$
is surjective.
\end{lem}
\begin{proof}
Since $\rank A = m$, 
$A(M_{n\times m}(K))=M_{m}(K)$ holds.
Thus we have
\begin{align*}
\Psi_A(M_{m\times n}(K))
&= 
\left\{ AN^t + NA^t \mid N\in M_{m\times n} (K)\right\} 
\\ &= 
\left\{ S + S^t\mid S\in M_{m} (K)\right\}
\\ &=
\begin{cases}
\Sym_m(K)&\text{if $\ch K\neq 2$,}\\
\Alt_m(K)&\text{if $\ch K= 2$.}
\end{cases}
\end{align*}
For $\ch K = 2$,
\begin{align*}
\Sym_{m}(K) &=
\{ R + R^t + \Diag (R)\mid R \in M_{m}(K)\; \mbox{is upper triangular}\}
\\ &\subset
\{ S + S^t + \Diag (S)\mid S \in M_{m}(K)\} 
\\ &= 
\{ AN^t + NA^t + \Diag(AN^t)\mid N \in M_{m \times n}(K) \}
\\ &= \Phi_A \left( M_{m\times n} (K)\right)\subset \Sym_{m}(K). 
\end{align*}
Thus $\Phi_A\left( M_{m\times n} (K)\right)=\Sym_{m}(K)$. 
\end{proof}

\begin{lem}\label{sumofmap}
Let $\ch K=2$, $\rank A=m$ for $A\in M_{m\times n} (K)$ and the vector $\allone$ does not belong to the row space of $A$. Define the map 
$$\begin{array}{ccl}
\alpha  : M_{m\times n} (K) & \rightarrow & K^{m} \\
N & \mapsto & \allone N^t.
\end{array}$$
Then the map 
$$\begin{array}{cccl}
\Phi_A \oplus \alpha : & M_{m\times n} (K) & \rightarrow & \Sym_m(K)\oplus K^m \\
  & N & \mapsto & \left( AN^t + NA^t + \Diag(AN^t),\allone N^t \right)
\end{array}$$
is surjective.
\end{lem}
\begin{proof}
Since $\Phi_A$ is surjective, it suffices
to show that the restriction of $\alpha$ to $\kernel\Phi_A$ is
surjective.
By the assumption, 
\[
\kernel A\supsetneq\kernel\begin{bmatrix} \allone\\ A\end{bmatrix}.
\]
This implies
\[
K=\{\allone x^t\mid x\in K^n,\; Ax^t=0\},
\]
and hence
\begin{align*}
\alpha(\kernel\Phi_A)
&\supseteq
\{\allone N^t\mid N\in M_{m\times n}(K),\;AN^t=0\}
\\ &=K^m.
\end{align*}
\end{proof}

\section{Codes over $\ZZ_{p^2}$ with prescribed residue and torsion} \label{sec4}

Throughout this section, we let $p$ be a prime, 
and $\ccode_1,\ccode_2$ codes of length $n$ over
$\ZZ_{p}$ such that $\ccode_1$ has generator matrix
\begin{equation}\label{C1}
\begin{bmatrix} 
I&A
\end{bmatrix},
\end{equation}
$\ccode_2$ has generator matrix
\begin{equation}\label{C2}
\begin{bmatrix} 
I&A\\ 0&B
\end{bmatrix},\end{equation}
$A\in M_{k_1\times (n-k_1)}(\ZZ)$, 
$B\in M_{k_2\times (n-k_1)}(\ZZ)$ 
and $\dim\ccode_1=k_1$, $\dim\ccode_2=k_1+k_2$. 

A code $\ccode$ of length $n$ over $\ZZ_{p^2}$ is free if 
$\ccode$ is isomorphic to the direct sum of copies of $\ZZ_{p^2}$, as a $\ZZ_{p^2}$-module, or equivalently,
$\pi(\ccode)=\iota^{-1}(\ccode)$. We first compute the number of free self-orthogonal codes with given residue, then compute the number of self-orthogonal codes with given residue and torsion.

\begin{lem}\label{genmatpure}
If $\ccode$ is a code of length $n$ over $\ZZ_{p^2}$ satisfying $\pi (\ccode) =\ccode_1$ and $\iota^{-1}(\ccode)=\ccode_2$,
then there exists a matrix 
$N\in M_{k_1\times (n-k_1)}(\ZZ)$ such that
\begin{equation}\label{Cgenmat}
\begin{bmatrix}
I&A+pN\\
0 & pB
\end{bmatrix}
\end{equation}
is a generator matrix of $\ccode$. Moreover, if $k_2 =0$, that is, $\pi (\ccode) =\iota^{-1}(\ccode)=\ccode_1$, such a
matrix $N$ is unique modulo $p$.
\end{lem}
\begin{proof}
Since $\pi (\ccode ) = \ccode_1$ and $\iota^{-1}(\ccode)=\ccode_2$,
$$ \ccode \supset \ZZ_{p^2}^{k_1+k_2} \begin{bmatrix} I + p M_1 & A + p M_2 \\0 & pB\end{bmatrix}
$$ 
for some $M_1\in M_{k_1} (\ZZ)$, $M_2\in M_{k_1\times (n-k_1)} (\ZZ)$. Then,
\begin{align*}
\ccode &\supset \ZZ_{p^2}^{k_1 + k_2} \begin{bmatrix} I - p M_1 & 0 \\ 0 & I
\end{bmatrix} \begin{bmatrix} I + p M_1 & A +p M_2  \\
0 & pB\end{bmatrix}
\\ &= \ZZ_{p^2}^{k_1+k_2} \begin{bmatrix} I & A + p(M_2 - M_1 A) \\
0 & pB \end{bmatrix}.
\end{align*}
Taking $N=M_2 -M_1 A$, we have
\begin{align*}
|\ccode| &\geq \left|\ZZ_{p^2}^{k_1+k_2} \begin{bmatrix} I & A + pN  \\
0 & pB\end{bmatrix}\right|
\\ &= p^{2k_1+k_2} 
\\ &=
|\ccode_2||\ccode_1|
\\ &=
|\ccode | 
\end{align*}
by (\ref{cardC}).
Thus, $\ccode$ has a generator matrix (\ref{Cgenmat}). It remains to show that for $k_2=0$, the matrix $N$ in (\ref{Cgenmat}) is determined uniquely by $\ccode$. Indeed, if 
$N_1, N_2 \in M_{k_1\times(n-k_1)} (\ZZ)$ and
\begin{equation*}\ZZ_{p^2}^{k_1} \begin{bmatrix} I & A+ pN_1 \end{bmatrix} = \ZZ_{p^2}^{k_1} \begin{bmatrix} I & A+ pN_2 \end{bmatrix},\end{equation*} then $A+pN_{1} \equiv A+pN_{2} \pmod{p^2}$, and hence $N_{1} \equiv N_{2}  \pmod{p}$.
\end{proof}

For the remainder of this section, we assume $\ccode_1 \subset \ccode_2 \subset \ccode_1^{\perp}$. This implies
\begin{equation}\label{soC1}
I + AA^t \equiv 0 \pmod{p}
\end{equation}
and
\begin{equation}\label{soC1inC2}
AB^t \equiv 0 \pmod{p}.
\end{equation}
Moreover, when $p=2$, we further assume
$\ccode_1$ to be doubly even, or equivalently, 
\begin{equation}\label{deC1}
\diag(I+AA^t)\equiv 0 \pmod{4}.
\end{equation}
This is justified by Lemma~\ref{DEres}.

\begin{lem}\label{C1=C2}
The number of free self-orthogonal codes $\ccode\subseteq \ZZ_{p^2}^n$ such that $\pi (\ccode) = \iota ^{-1}(\ccode)=\ccode_1$ is $p^{k_{1}(2n-3k_{1}-1)/2}$ for odd primes $p$ and $2^{k_{1}(2n-3k_{1}+1)/2}$ for $p = 2$. 
\end{lem}
\begin{proof}
By Lemma~\ref{genmatpure}, a code $\ccode\subset \ZZ_{p^2}^{n}$ such that $\pi (\ccode) = \iota^{-1}(\ccode) = \ccode_1$ has a generator matrix $\begin{bmatrix} I & A+pN \end{bmatrix}$ for some matrix $N\in M_{k_1 \times (n-k_1)}(\ZZ)$ which is unique modulo $p$. Observe that $\ccode$ is self-orthogonal if and only if
\begin{equation}\label{modp2} I + A A^t + p(AN^t + NA^t) \equiv 0 \pmod{p^2}.\end{equation}
By (\ref{soC1}),
$\displaystyle -\frac{1}{p}(I + A A^t)\bmod{p}$ is a symmetric matrix over $\ZZ_p$ if $p$ is odd, and by (\ref{deC1}), it is an alternating matrix over $\ZZ_p$ if $p=2$. Therefore by Lemma~\ref{rankA},
\begin{align*}
&|\{\ccode\mid \ccode \text{ is self-orthogonal, }\pi (\ccode) = \iota ^
{-1}(\ccode)=\ccode_1 \}| 
\\ &=
\left|\{ N \bmod{p}\mid N\in M_{k_1\times (n-k_1)}(\ZZ)\;\mbox{satisfies}\; (\ref{modp2}) \}\right| 
\\ &= 
\left|\left\{ N \mid N\in M_{k_{1}\times (n-k_{1})}(\ZZ_{p})\;\mbox{and}\; \Psi_A (N) = -\frac{1}{p}(I + A A^t)\bmod{p} \right\}\right| 
\\ &=
\left| \Psi_A ^{-1} \left( - \frac{1}{p}(I + A A^t)\bmod{p} \right) \right| 
\\ &= 
|\kernel \Psi_A | 
\\ &= 
\left\{\begin{array}{cl} 
\displaystyle \frac{|M_{k_1\times (n-k_1)}(\ZZ_{p})|}{|\Sym_{k_1}(\ZZ_{p})|} & \text{if}\;p\;\text{is odd},   \\\\
\displaystyle \frac{|M_{k_1\times (n-k_1)}(\ZZ_{2})|}{|\Alt_{k_1}(\ZZ_{2})|} & \text{if}\;p = 2, \end{array} \right.
\\ &=
\begin{cases}
p^{k_1(n-k_1)-k_1(k_1+1)/2} & \text{if $p$ is odd,}  \\
2^{k_1(n-k_1)-k_1(k_1-1)/2} & \text{if $p = 2$.} 
\end{cases}
\end{align*}
\end{proof}

Let us consider sets
\begin{align*}
X &= \{\ccode\mid \ccode \text{ is self-orthogonal, }\pi (\ccode) = \iota ^
{-1}(\ccode)=\ccode_1 \} \; \mbox{and} \\
X'&= \{ \ccode ' \mid \ccode' \text{ is self-orthogonal, } \pi (\ccode ')=\ccode_1, \iota ^{-1}(\ccode ') = \ccode_2 \}.
\end{align*}
By Lemma~\ref{C1=C2}, we have $|X|=p^{k_{1}(2n-3k_{1}-1)/2}$ if $p$ is odd and $|X|=2^{k_{1}(2n-3k_{1}+1)/2}$ if $p=2$.

\begin{lem}\label{unique}
If $\ccode\in X$, then there exists a unique $\ccode'\in X'$ containing $\ccode$.
\end{lem}
\begin{proof}
By Lemma~\ref{genmatpure}, $\ccode$ has a generator matrix
$\begin{bmatrix} I & A + pN \end{bmatrix}$ for some $N$. Then the code
$\ccode_{0}'$ with generator matrix
$$\begin{bmatrix} I & A + pN\\ 0 & p B  \end{bmatrix}$$ satisfies $\pi(\ccode_{0}')=\ccode_1$ and $\iota^{-1}(\ccode_{0}')=\ccode_2$.
Since $\ccode\in X$, (\ref{soC1inC2}) implies that $\ccode_{0}'$ is self-orthogonal, hence $\ccode_{0}'\in X'$. If 
$\ccode'\in X'$ and $\ccode\subset\ccode'$, then Lemma~\ref{genmatpure} implies $\ZZ_{p^2}^{k_2}\begin{bmatrix} 0 & pB \end{bmatrix}\subset\ccode'$. This, together with $\ccode
\subset\ccode'$ forces $\ccode_{0}'
\subset\ccode'$, and hence $\ccode_{0}'=\ccode'$.
\end{proof}

\begin{lem}\label{Ms}
Let $\ccode'\in X'$. Then $|\{\ccode\in X \mid \ccode \subset\ccode'\}|=p^{k_{1}k_{2}}$.  
\end{lem}
\begin{proof}
Let $\begin{bmatrix} I & A + pN \\
0 & p B \end{bmatrix}$ be a generator matrix of $\ccode'$. Consider the map
$$\begin{array}{ccl}
\varphi : M_{k_1\times k_2} (\ZZ_p) & \rightarrow & \{ \ccode \in X \mid \ccode \subset \ccode' \} \\\\
M \bmod p & \mapsto & \ZZ_{p^2}^{k_1}\begin{bmatrix} I & A+p(N + MB)\end{bmatrix}.
\end{array}$$
We claim that this map is bijective. Suppose $M_1,M_2\in M_{k_1\times k_2} (\ZZ)$, 
$\varphi(M_{1})= \varphi(M_{2})$. Then we have 
$A+p(N + M_{1}B) \equiv A+p(N + M_{2}B) \pmod{p^2}$.
Since $\rank(B\bmod{p})=k_2$, we conclude $M_{1} \equiv M_{2} \pmod{p}$,
which shows that $\varphi$ injective. Next we show that $\varphi$ is surjective. Suppose $\ccode\in X$ and $\ccode\subset\ccode'$. Then by Lemma~\ref{genmatpure}, $\ccode = \ZZ_{p^2}^{k_1}\begin{bmatrix} I & A+ pF \end{bmatrix}$ for some matrix $F$. Since $\ccode\subset\ccode'$,
$A+pF\equiv A + pN + pMB \pmod{p^2}$ for some matrix $M$. Then we have $F\equiv N+MB\pmod{p}$, so we conclude that $\varphi$ is surjective. Therefore, the map $\varphi$ is bijective. It follows that the number of codes $\ccode\in X$ contained in $\ccode'$ is $p^{k_{1}k_2}$. 
\end{proof}

\begin{thm}
\label{C1subC2} Let $\ccode_1$ and $\ccode_2$ be codes of length $n$ over $\ZZ_{p}$ where $\ccode_1\subset \ccode_2\subset\ccode_1 ^{\perp}$. Assume further that $\ccode_1$ is doubly even when $p=2$. If $\dim \ccode_1 = k_1$ and $\dim \ccode_2 =k_1 + k_2$, then the number of self-orthogonal codes $\ccode$ of length $n$ over $\ZZ_{p^2}$ such that $\pi (\ccode) = \ccode_1$ and $\iota ^{-1}(\ccode)=\ccode_2$ is $\displaystyle p^{k_1(2n-3k_1-1-2k_2)/2}$ for odd primes $p$ and $\displaystyle 2^{k_1(2n-3k_1+1-2k_2)/2}$ for $p = 2$. 
\end{thm}
\begin{proof}
We may assume without loss of generality that $\ccode_1$ and $\ccode_2$ are codes with generator matrices given by 
(\ref{C1}) and (\ref{C2}), 
respectively.
By Lemma~\ref{unique} and Lemma~\ref{Ms}, we have
\begin{align*}
|X'|&=
p^{-k_1k_2}\sum_{\ccode'\in X'}
|\{\ccode\in X\mid\ccode\subset\ccode'\}|
\\ &=
p^{-k_1k_2}\sum_{\ccode\in X}
|\{\ccode'\in X'\mid\ccode\subset\ccode'\}|
\\ &=
p^{-k_1k_2}|X|.
\end{align*}
The result then follows from Lemma~\ref{C1=C2}.
\end{proof}

\section{Quaternary even codes with prescribed residue and torsion}

We are mainly concerned with the enumeration of codes containing the vector
$\allone$, or codes containing a vector each of whose
coordinate is $1$ or $-1$. This restriction forces the length of a code to be a multiple of $8$.
First, let us consider such codes containing $\allone$. Let $n\in 8\ZZ$, $\displaystyle 1\leq k_1\leq \frac
{n}{2}$, $\ccode_1, \ccode_2$ binary codes of length $n$ 
such that $\ccode_1$ is doubly even and has generator matrix
\begin{equation}\label{C1E}
\begin{bmatrix}
1&\allone&\allone\\
0&I_{k_1-1}&A
\end{bmatrix},
\end{equation}
$\ccode_2$ has generator matrix
\begin{equation}\label{C2E}
\begin{bmatrix}
1&\allone&\allone\\
0&I_{k_1-1}&A \\
0&0&B 
\end{bmatrix},
\end{equation}
$A\in M_{(k_{1}-1)\times (n-k_1)}(\ZZ)$, 
$B\in M_{k_2\times (n-k_1)}(\ZZ)$ and $\dim\ccode_1=k_1$, $\dim\ccode_2=k_1+k_2$. 
Moreover, we assume that $\ccode_1\subset\ccode_2\subset\ccode_1^{\perp}$ and $\ccode_1$ is doubly even.
Then the matrices $A$ and $B$ satisfy (\ref{soC1})--(\ref{deC1}), and
\begin{equation}\label{evenB}
\allone B^t \equiv 0 \pmod{2}.
\end{equation}
We may assume without loss of generality that the entries of the matrix $A$ are all $0$ or $1$. Then
\begin{equation}\label{devenC1}
\allone + \allone A^{t} \equiv 0 \pmod{4}.
\end{equation}
\begin{lem}\label{allone}
The vector $\allone$ does not belong to the row space
of the matrix $A$ over $\ZZ_2$.
\end{lem}
\begin{proof}
If $k_1 =1$, the assertion is trivial, so assume $k_1 >1$. Suppose that $\allone$ is in the row space of $A$ over $\ZZ_{2}$. Then $bA\equiv \allone\pmod{2}$ for some vector $b\in \ZZ^{k_1 -1}$. By (\ref{soC1}) and
(\ref{devenC1}), we have
\begin{align*}
0 & \equiv \begin{bmatrix}b(\allone + \allone A^{t})^t & b(I + AA^{t})\end{bmatrix} \pmod{2}
\\ & \equiv 
\begin{bmatrix}b\allone ^{t} + \allone \allone^{t} & b + \allone A^{t}\end{bmatrix} \pmod{2}
\\ & \equiv 
\begin{bmatrix}b\allone ^{t} + (n-k_1) & b + \allone\end{bmatrix} \pmod{2} &&\text{(by (\ref{devenC1}))}.
\end{align*}
Therefore, $b\equiv \allone \pmod{2}$ and 
\[
b\allone ^{t} + (n-k_1) \equiv ( k_1 -1) + (n-k_1) \not \equiv 0 \pmod{2}. 
\]
This is a contradiction.
\end{proof}

\begin{lem}\label{TypeIISON}
Define the maps
$$\begin{array}{ccl}
\alpha  : M_{(k_1 -1)\times (n-k_1)} (\ZZ_2) & \rightarrow & \ZZ_{2}^{k_1 -1} \\
N \bmod{2}  & \mapsto & \allone N^t \bmod{2}
\end{array}$$
and
$$\begin{array}{cccl}
\Phi_A : & M_{(k_1 -1)\times (n-k_1)} (\ZZ_2) & \rightarrow & \Sym_{k_1 -1}(\ZZ_2) \\
  & N \bmod{2} & \mapsto & AN^t + NA^t + \Diag(AN^t) \bmod{2}.
\end{array}$$
For $N\in M_{(k_1-1)\times (n-k_1)}(\ZZ)$, let 
\[
\ccode=\ZZ_4^{k_{1}}
\begin{bmatrix}
1&\allone&\allone\\
0&I_{k_{1}-1}&A+2N
\end{bmatrix}.
\]
Then $\ccode$ is an even code if and only if $N \bmod{2}$ belongs to the set 
\[
\kernel\alpha\cap \Phi_A^{-1}\left(\frac{1}{2}(I + AA^t)+\frac{1}{4} (I+\Diag AA^t)
\bmod{2}\right).
\]
\end{lem}
\begin{proof}
$\ccode$ is self-orthogonal if and only if
\begin{align}
\allone+\allone A^t+2\allone N^t\equiv0\pmod{4},
\label{eq:s1}\\
I+AA^t+2(AN^t+NA^t)\equiv0\pmod{4}.\label{eq:s2}
\end{align}
In view of (\ref{devenC1}), (\ref{eq:s1}) is equivalent to $N\bmod{2}\in\kernel \alpha$.
Moreover, when $\ccode$ is self-orthogonal, 
$\ccode$ is even if and only if
\begin{equation}\label{eq:d1}
\diag(I+AA^t+2(AN^t+NA^t)+4NN^t)\equiv0\pmod{8}.
\end{equation}
Since $\diag NN^t\equiv\allone N^t\equiv 0\pmod{2}$, this is equivalent to
\begin{equation}\label{eq:d4}
\Diag(AN^t)\equiv \frac{1}{4} (I +\Diag AA^t)
\pmod{2}.
\end{equation}
Therefore, the code $\ccode$ is even if and only if $N\bmod{2}\in \kernel \alpha$ and conditions (\ref{eq:s2}) and (\ref{eq:d4}) are
satisfied. The latter two conditions are equivalent to
$$\Phi_A(N\bmod{2})=\frac{1}{2}(I + AA^t)+\frac{1}{4} (I+\Diag AA^t)
\bmod{2}.
$$
\end{proof}

\begin{lem}\label{withone}
The number of free quaternary even codes $\ccode$ of length $n$ containing $\allone$ such that $\pi (\ccode) = \iota ^{-1}(\ccode)=\ccode_1$ is $2^{(k_{1}-1)(2n-3k_{1}-2)/2}$. 
\end{lem}
\begin{proof}
By Lemma~\ref{TypeIISON}, we have
\begin{align*}
&|\{\ccode \mid \ccode \text{ is even, }\pi (\ccode) = \iota ^
{-1}(\ccode)=\ccode_1 , \allone\in\ccode \}| 
\\ &=
\left|\kernel \alpha 
\cap \Phi_A^{-1}\left( \frac{1}{2}(I + AA^t)+\frac{1}{4} (I +\Diag AA^t)
\bmod{2}\right) \right|
\\ &= 
\left|(\Phi_A\oplus\alpha)^{-1}\left(\left(\frac{1}{2}(I + AA^t)+\frac{1}{4} (I+\Diag AA^t)
\bmod{2}\right),0\right)\right|. 
\end{align*} 
By Lemma~\ref{allone}, $\allone$ does not belong to the row space of $A$ over $\ZZ_2$. 
Thus, the map $\Phi_A\oplus\alpha$ is surjective by Lemma~\ref{sumofmap}. 
Therefore, the number of even codes $\ccode$ containing $\allone$ such that 
$\pi (\ccode) = \iota ^
{-1}(\ccode)=\ccode_1$ is 
\begin{equation*}
|\kernel(\Phi_A\oplus\alpha)|=
2^{(k_{1}-1)(n-k_{1})-(k_1 -1)-k_{1}(k_{1}-1)/2}. 
\end{equation*} 
\end{proof}
Let us consider the sets 
\begin{align*}
Y &= \{ \ccode \mid \ccode\; \mbox{is quaternary even},\; \pi (\ccode)=\iota ^{-1}(\ccode) = \ccode_1 ,\; \allone\in\ccode\}\;\mbox{and} \\
Y' &= \{ \ccode' \mid \ccode'\; \mbox{is quaternary even},\; \pi (\ccode')=\ccode_1,\; \iota ^{-1}(\ccode') = \ccode_2, \; \allone\in\ccode' \}.
\end{align*}
By Lemma~\ref{withone}, $|Y|=2^{(k_{1}-1)(2n-3k_{1}-2)/2}$.
\begin{lem}\label{uniqueY}
If $\ccode\in Y$, then there exists a unique $\ccode'\in Y'$ containing $\ccode$.
\end{lem}
\begin{proof}
The unique code $\ccode'$ given in Lemma~\ref{unique} is even. The matrix $B\bmod{2}$ has even weight rows by (\ref{evenB}) so $2B\bmod{4}$ has rows with Euclidean weights divisible by $8$.
\end{proof}

\begin{lem}\label{MsY}
Let $\ccode'\in Y'$. Then $|\{\ccode\in Y \mid \ccode \subset\ccode'\}|=2^{(k_{1}-1)k_{2}}$.  
\end{lem}
\begin{proof}
Let $\begin{bmatrix}
1&\allone&\allone\\
0&I_{k_1-1}&A+2N \\
0&0&2B 
\end{bmatrix}$ be a generator matrix of $\ccode'$. Consider the map
$$\begin{array}{ccl}
\varphi' : M_{(k_1-1)\times k_2} (\ZZ_2) & \rightarrow & \{ \ccode \in Y \mid \ccode \subset \ccode' \} \\\\
M \bmod 2 & \mapsto & \ZZ_{4}^{k_1}\begin{bmatrix}
1&\allone&\allone\\
0&I_{k_1-1}&A+2(N+MB) \end{bmatrix}.
\end{array}$$
We claim that this map is bijective. Suppose $M_1,M_2\in M_{(k_1-1)\times k_2} (\ZZ)$ and 
$\varphi'(M_1)= \varphi'(M_2)$. Using same approach as in the proof of Lemma~\ref{Ms}, we can
show that the map $\varphi'$ is well-defined and injective with $M_1\equiv M_2 \pmod{2}$. Next we show that the map is surjective. Suppose $\ccode\in Y$ and $\ccode\subset\ccode'$, then $\ccode=\ZZ_{4}^{k_1}\begin{bmatrix} 1&\allone&\allone\\ 0 & I_{k_1-1} & A+ 2F \end{bmatrix}$ for some matrix $F$. Since $\ccode\subset\ccode'$, $A+2F\equiv A +2N +2MB\pmod{4}$
for some matrix $M$. Then we have $F\equiv N+MB\pmod{2}$, so we conclude that $\varphi'$ is
surjective. Hence, the map $\varphi'$ is bijective. Therefore the number of codes $\ccode\in Y$ contained in $\ccode'$ is $2^{(k_{1}-1)k_2}$. 
\end{proof}

\begin{thm}\label{DEallone}
Let $\ccode_1$ and $\ccode_2$ be binary codes of length $n$ where $\ccode_1\subset \ccode_2\subset\ccode_1 ^{\perp}$. If $\ccode_1$ is doubly even, $\allone\in \ccode_1$ and $\dim \ccode_1 = k_1$ and $\dim \ccode_2 =k_1 + k_2$, then the number of quaternary even codes $\ccode$ containing $\allone$ such that $\pi (\ccode) = \ccode_1$ and $\iota ^{-1}(\ccode)=\ccode_2$ is $\displaystyle 2^{(k_{1}-1)(2n-3k_{1}-2-2k_2)/2}$.
\end{thm}
\begin{proof}
We may assume without loss of generality that $\ccode_1$ and $\ccode_2$ are codes with generator matrices given by (\ref{C1E}) and (\ref{C2E}), respectively.
By Lemma~\ref{uniqueY} and Lemma~\ref{MsY}, we have
\begin{align*}
|Y'|&=
2^{-(k_1-1)k_2}\sum_{\ccode'\in Y'}
|\{\ccode\in Y\mid\ccode\subset\ccode'\}|
\\ &=
2^{-(k_1-1)k_2}\sum_{\ccode\in Y}
|\{\ccode'\in Y'\mid\ccode\subset\ccode'\}|
\\ &=
2^{-(k_1-1)k_2}|Y|.
\end{align*}
The result then follows from Lemma~\ref{withone}.
\end{proof}
\begin{thm}\label{doublyeven}
Let $\ccode_1$ and $\ccode_2$ be binary codes of length $n$ where $\ccode_1\subset \ccode_2\subset\ccode_1 ^{\perp}$. If $\ccode_1$ is doubly even, $\allone\in \ccode_1$, $\dim \ccode_1 = k_1$ and $\dim \ccode_2 =k_1 + k_2$, then the number of quaternary even codes $\ccode$ containing an element of $\{\pm 1\}^n$ such that $\pi (\ccode) = \ccode_1$ and $\iota ^{-1}(\ccode)=\ccode_2$ is $\displaystyle 2^{(n-k_1-k_2)+(k_{1}-1)(2n-3k_{1}-2-2k_2)/2}$ . 
\begin{proof}
Consider the sets 
\begin{align*}
Z &=\{\pm 1\}^n\;\mbox{and} \\
U &=\{ \ccode \mid \ccode\; \mbox{is quaternary even},\; \pi (\ccode)=\ccode_1,\;\iota ^{-1}(\ccode) = \ccode_2,\; Z \cap \ccode \neq
\emptyset \}.
\end{align*}
If $\ccode\in U$, then there exists $z\in Z \cap \ccode $. Then $Z\cap \ccode = \{ z + \iota(x) \mid x \in \ccode_2 \}$, hence $|Z\cap \ccode | = |\ccode_2|$.
Therefore, we have
\begin{align*}
2^{k_1 + k_2} |U| &= |\ccode_2||U| 
\\ &=
\sum_{\ccode \in U} |Z \cap \ccode|
\\ &=
\sum_{z \in Z}|\{\ccode \in U \mid z\in \ccode \}|
\\ &=
\sum_{z \in Z}|\{\ccode \in U \mid \allone\in \ccode \}|
\\ &=
|Z|2^{(k_{1}-1)(2n-3k_{1}-2-2k_2)/2} && \text{(by Theorem~\ref{DEallone})} 
\\ &= 2^{n+(k_{1}-1)(2n-3k_{1}-2-2k_2)/2}.
\end{align*}
\end{proof}
\end{thm} 

\section{Main Results}

Let $\sigma(n,k_1)$ denote the number of distinct doubly even binary codes of length $n$ and dimension $k_1$, and let $\sigma_{p}(n,k_1)$ be the number of distinct self-orthogonal $p$-ary codes, with $p$ an odd prime, of length $n$ and dimension $k_1$.  Let $\sigma_{\allone}(n,k_1)$ be the number of distinct doubly even binary codes of length $n$ and dimension $k_1$ containing $\allone$. The value of $\sigma _{p}(n,k_1)$ is given in \cite{iso,Golay}. The values of $\sigma(n,k_1)$ and $\sigma _{\allone}(n,k_1)$ are given in \cite{Gaborit}. All of these values can also be derived from \cite{Chaudhuri}. 
For $ k\leq n$, we define the Gaussian coefficient $\begin{bmatrix} n \\ k \end{bmatrix}_{p}$ as
$$\begin{bmatrix} n \\ k \end{bmatrix}_{p}=\frac{(p^{n}-1)(p^{n}-p)\cdots (p^{n}-p^{k-1})}{(p^{k}-1) (p^{k}-p)\cdots (p^{k}-p^{k-1})}.$$

\begin{cor}\label{cor1}
The number of distinct self-orthogonal codes of length $n$ over $\ZZ_{p^2}$ of type $\{k_1, k_2\}$ is
\begin{equation}\label{main_p_odd}
\displaystyle M_{p^2}(k_1,k_2) = \sigma _{p}(n,k_1)\:\begin{bmatrix} n-2k_1 \\ k_2 \end{bmatrix}_{p} \:\displaystyle p^{k_1(2n-3k_1-1-2k_2)/2},
\end{equation}
for odd primes $p$, and
\begin{equation}\label{main_p=2}
\displaystyle M_{4}(k_1,k_2) = \sigma(n,k_1)\: \begin{bmatrix} n-2k_1 \\ k_2 \end{bmatrix}_{2} \:\displaystyle 2^{k_1(2n-3k_1+1-2k_2)/2},
\end{equation}
for $p=2$.   
\end{cor}
\begin{proof}
Given a self-orthogonal $[n,k_1]$ code $\ccode_1$, we have $\begin{bmatrix} n-2k_1 \\ k_2 \end{bmatrix}_{p}$ codes $\ccode_2$ such that $\ccode_1\subseteq\ccode_2\subseteq\ccode_1 ^{\perp}$. The result follows from Theorem~\ref{C1subC2}. 
\end{proof}
As an example, we consider the case $n=4$, $k_1=k_2=1$, and $p=3$. Let $\ccode_1$, $\ccode_2$, $\ccode_3$, $\ccode_4$ be the self-orthogonal codes over $\ZZ_9$ of type $\{1, 1\}$ with generator matrices
$$
 \begin{bmatrix} 1&1&4&0 \\
                        0&3&6&0 \end{bmatrix},
    		\begin{bmatrix} 1&1&4&3 \\
    		                0&3&6&0 \end{bmatrix},
    		\begin{bmatrix} 1&1&4&6 \\
                        0&3&6&3 \end{bmatrix},
    		\begin{bmatrix} 1&7&7&0 \\
                        0&0&0&3 \end{bmatrix},$$
respectively. The orders of their automorphism groups are $24,12,4$ and $8$, respectively. Thus, we have
\begin{align*}
\displaystyle 
\sum_{i=1}^{4}\frac{|E|}{|\Aut\ccode_i|} &= \frac{2^4 4!}{24} + \frac{2^4 4!}{12} + \frac{2^4 4!}{4} + \frac{2^4 4!}{8}
&= 16 + 32 + 96 + 48 =192. 
\end{align*}
Using Corollary~\ref{cor1}, we have
\begin{align*}
\displaystyle M_{9}(1,1) &= \sigma _{3}(4,1)\:\begin{bmatrix} 4-2 \\ 1 \end{bmatrix}_{3} \:\displaystyle 3^{3-1-1}
\\ &=
16 \cdot 4\cdot 3 = 192.
\end{align*}
This implies that $\{ \ccode_1, \ccode_2, \ccode_3, \ccode_4 \}$ is a complete set of representatives for equivalence classes of self-orthogonal codes of length $4$ and type $\{1, 1\}$ over $\ZZ_9$.

As a consequence of Corollary~\ref{cor1}, we have
\begin{cor}\label{cor2}
The number of distinct self-dual codes over $\ZZ_{p^2}$ of length $n$ is
\begin{equation}\label{sd}
\displaystyle\sum_{0\leq k_1\leq \left\lfloor \frac{n}{2}
\right\rfloor} M_{p^2}(k_1,n-2k_1).\end{equation}
\end{cor}
\noindent Corollary~\ref{cor2} agrees with the results in \cite{Gaborit}, \cite{Z9}, and \cite{Zp2}.
\begin{cor}\label{cor3}
The number of distinct quaternary even codes of length $n$ containing $\allone$, of type $\{k_1, k_2\}$ is  
\begin{equation}\label{main_1}
\displaystyle D_{\allone} (k_1,k_2)= \sigma _{\allone}(n,k_1)\: \begin{bmatrix} n-2k_1 \\ k_2 \end{bmatrix}_{2}\:\displaystyle 2^{(k_{1}-1)(2n-3k_{1}-2-2k_2)/2} ,
\end{equation}
and the number of distinct quaternary even codes of length $n$ containing an element of $\{\pm 1\}^n$, of type $\{k_1, k_2\}$ is
\begin{equation}\label{main_pm1}
\displaystyle D (k_1,k_2) = \sigma _{\allone}(n,k_1)\: \begin{bmatrix} n-2k_1 \\ k_2 \end{bmatrix}_{2} \:\displaystyle 2^{(n-k_1-k_2)+(k_{1}-1)(2n-3k_{1}-2-2k_2)/2}.
\end{equation} 
\end{cor}
\begin{proof}
Given a binary doubly even $[n,k_1]$ code $\ccode_1$ containing $\allone$, we have $\begin{bmatrix} n-2k_1 \\ k_2 \end{bmatrix}_{2}$ codes $\ccode_2$ such that $\ccode_1\subseteq\ccode_2\subseteq\ccode_1 ^{\perp}$. Thus, (\ref{main_1}) follows from Theorem~\ref{DEallone}. Similarly, (\ref{main_pm1}) follows from Theorem~\ref{doublyeven}. 
\end{proof}
And as a consequence of Corollary~\ref{cor3}, we have
\begin{cor}\label{cor4}
The number of distinct quaternary Type II codes of length $n$ containing $\allone$ is 
\begin{equation}\label{TypeIIone}
\displaystyle \displaystyle\sum_{0\leq k_1\leq \frac{n}{2}}
D_{\allone}(k_1,n-2k_1),\end{equation}
and the number of distinct quaternary Type II codes of length $n$ containing an element of $\{\pm 1\}^n$ is
\begin{equation}\label{TypeIIpm}
\displaystyle \displaystyle\sum_{0\leq k_1\leq \frac{n}{2}}
D(k_1,n-2k_1).\end{equation}
\end{cor}
\noindent The formula (\ref{TypeIIpm}) agrees with \cite[Theorem 6]{Gaborit}.



\end{document}